\documentclass[psamsfont,a4paper,12pt]{amsart}

\usepackage[english]{babel}   %loads english

\usepackage{setspace}     %controls line-spacing
\usepackage{enumerate} 	%allows lists
\usepackage{afterpage} 	%keeps tables and figures where they belong
\usepackage[usenames,dvipsnames]{color}
\usepackage{graphicx}	%allows inclusion of pictures

\usepackage{amsmath} 	%just apparently the most important thing ever
\usepackage{amssymb}	%symbols in math
\usepackage{mathrsfs} 	%really fancy script font
\usepackage{amsfonts} 	%standard fonts
\usepackage{latexsym} 	%more symbols
\usepackage[latin1]{inputenc} %sets up proper special characters

\usepackage{amsthm} 	%sets up theorem styling
\usepackage{stackrel}       %allows stacking in mathmode  
\usepackage{amscd} 	%commutative diagrams                                 

\usepackage{tabularx}	%automatically widening tables
\usepackage{longtable}	%for tables longer than a page

\usepackage{verbatim} %comment environment?
\usepackage{blindtext} %something about dummy text

\usepackage{lipsum}

\allowdisplaybreaks

\newcommand{\R}{\mathbb{R}} 
 
\newcommand{\NN}{\mathbb{N}}

\renewcommand{\to}{\longrightarrow}

\newtheorem{Thm}{Theorem}[section]		%%with numbering
\newtheorem{Lemma}[Thm]{Lemma}

\theoremstyle{definition}

\theoremstyle{remark}
 
\newtheorem*{rmk}{Remark}

\newtheorem{ind}[]{{\rm\it Indice}}

\renewcommand{\epsilon}{\varepsilon}

\usepackage{multicol}

\usepackage{tikz}
\usetikzlibrary{arrows}

\usepackage{paralist}
\usepackage{cite}

\title{Hyperbolicity of the partition Jensen polynomials}
\author[Larson]{Hannah Larson}
\author[Wagner]{Ian Wagner}
\usepackage{listings}

\begin{document}
\numberwithin{equation}{section}

\maketitle

 \begin{abstract}
 Given an arithmetic function $a: \NN \rightarrow \R$, one can associate a naturally defined, doubly infinite family of Jensen polynomials. Recent work of Griffin, Ono, Rolen, and Zagier shows that for certain families of functions $a: \NN \rightarrow \R$, the associated Jensen polynomials are eventually hyperbolic (i.e., eventually all of their roots are real). This work proves Chen, Jia, and Wang's conjecture that the partition Jensen polynomials are eventually hyperbolic as a special case. Here, we make this result explicit. Let $N(d)$ be the minimal number such that for all $n \geq N(d)$, the partition Jensen polynomial of degree $d$ and shift $n$ is hyperbolic. We prove that $N(3)=94$, $N(4)=206$, and $N(5)=381$, and in general, that $N(d) \leq (3d)^{24d} (50d)^{3d^{2}}$.
 \end{abstract}

\section{Introduction}
Given a function $a: \NN \to \R$ and positive integers $d$ and $n$, the associated \textit{Jensen polynomial of degree $d$ and shift $n$} is defined by
\[J^{d,n}_a(X) := \sum_{j=0}^d {d \choose j} a(n+j) X^j.\]
A polynomial is said to be \textit{hyperbolic} if all of its zeros are real.
Given an entire real function $\varphi(x)$ with Taylor expansion $\varphi(x) = \sum_{n \geq 0} \frac{\alpha(n)x^n}{n!}$, it is a theorem of Jensen \cite{J} that $\varphi(x)$ is in the Laguerre-P\'olya class if and only if all of the associated Jensen polynomials $J_\alpha^{d,0}(X)$ are hyperbolic. P\'olya proved \cite{P}  that the Riemann Hypothesis is equivalent to the hyperbolicity of all Jensen polynomials associated to Riemann's $\xi(s)$.

In this paper, we study the hyperbolicity of Jensen polynomials $J_p^{d,n}(X)$ associated to the partition function $p(n)$, which counts the number of integer partitions of $n$.  Chen, Jia, and Wang conjectured that for each positive integer $d$, $J_{p}^{d,n}(X)$ is eventually hyperbolic \cite{CJW}. For example, hyperbolicity of $J^{2,n}_p(X)$ is equivalent to $p(n+2)p(n) \leq p(n+1)^2$, a condition known as log concavity. Nicolas originally proved that this condition holds for all $n \geq 25$ in \cite{Ni}.  This result was reproved by Desalvo and Pak in \cite{DP}.

Recent results of Griffin, Ono, Rolen, and Zagier \cite{GORZ} show that Jensen polynomials for a large family of functions, including those associated to $\xi(s)$ and the partition function, are eventually hyperbolic. 
Their proof relates the polynomials $J^{d,n}_p(X)$ to the \textit{Hermite polynomials} $H_d(X)$, defined by the generating function
\[e^{tX-t^2} = \sum_{d=0}^\infty H_d(X) \cdot \frac{t^d}{d!} = 1 + X \cdot t + (X^2 - 2) \cdot \frac{t^2}{2} + (X^3 - 6 X) \cdot \frac{t^3}{6} + \ldots .\]
More precisely, if
\[c := \frac{2}{3}\pi^2, \qquad w(n) := \frac{1}{\sqrt{c(n-\frac{1}{24})}}, \qquad \delta(n) := \frac{c w(n)^{\frac{3}{2}}}{\sqrt{2}},\]
the authors prove that
\begin{equation} \label{grozmain}
\lim_{n \rightarrow \infty} \frac{2^d}{p(n)\delta(n)^d} \cdot J^{d,n}_p\left(\delta(n) X - e^{-cw(n)/2}\right) = H_d(X). 
\end{equation}
Since the Hermite polynomials have distinct real roots, it follows that the polynomial on the left-hand side above, and hence $J^{d,n}_p(X)$, is eventually hyperbolic. In other words, for each $d$ there exists some $N$ such that for all $n \geq N$, the polynomial $J^{d,n}_p(X)$ is hyperbolic. Define $N(d)$ to be the minimal such $N$. For example, the results of Nicolas and Desalvo and Pak show $N(2) = 25$. We determine the following further values of $N(d)$.

\begin{Thm} \label{lowdegrees}
Let $N(d)$ be defined as above. Then $N(3) = 94, N(4) = 206,$ and $N(5) =381$.
\end{Thm}

\begin{rmk}
During the preparation of this paper, the authors were notified that Chen, Jia, and Wang \cite{CJW} independently proved $N(3)=94$ using different methods.
\end{rmk}

The proof of Theorem \ref{lowdegrees} relies on obtaining functions that closely approximate the ratios $p(n+j)/p(n)$
and bounding the error of these approximations for large $n$.
 For $d=3,4,5$, direct computation gives rise to good bounds, allowing us to reduce Theorem \ref{lowdegrees} to checking a reasonably small finite number of cases. 
 As an illustration of these techniques, we also prove a recent conjecture of Chen which involves an inequality of polynomials in ratios of close partition numbers.

\begin{Thm}[Conjecture $6.13$ in \cite{Chen}] \label{con}
Let $u_n = p(n+1)p(n-1)/p(n)^2$. Then for all $n \geq 2$, we have
\begin{equation} \label{chen}
4(1-u_n)(1-u_{n+1}) < \left(1 + \frac{\pi}{\sqrt{24}n^{3/2}}\right)(1 - u_n u_{n+1})^2.
%\frac{4\left(p(n)^2 - p(n+1)p(n-1)\right)\left(p(n+1)^2 - p(n+2)p(n)\right)}{\left(p(n)^2p(n+1)^2 - p(n-1)p(n)p(n+1)p(n+2)\right)^2} < 1 + \frac{\pi}{\sqrt{24}n^{3/2}}.
\end{equation}
\end{Thm}

 For arbitrary $d$, similar techniques, along with the convergence of $J^{d,n}_p(X)$ to the Hermite polynomials $H_d(X)$ after change of variable, gives rise to an upper bound for $N(d)$. However, without the benefit of direct computation we rely on rather rough estimates for the errors mentioned above. This yields the following.

\begin{Thm} \label{general}
For every positive integer $d$, we have $N(d) \leq (3d)^{24d} (50d)^{3d^2}$.
\end{Thm}

This paper is organized as follows.  In Section $2$ we describe an equivalent condition for a polynomial to have all real roots and prove two lemmas that bound higher order terms discarded by the methods in \cite{GORZ}.  In Section $3$ we prove Theorem \ref{lowdegrees} 
and in Section $4$ we prove Theorem \ref{general} through a series of estimates on accumulating error terms.  The appendix contains the Mathematica and Sage code used in the proof of Theorem \ref{lowdegrees}.

\subsection*{Acknowledgements}
The authors thank Ken Ono for suggesting this problem and providing advice. The authors are also grateful to Jesse Thorner for his help implementing Mathematica code that was used in the proof of Theorem \ref{lowdegrees}. This research was supported by the National Science Foundation under Grant 1557960.

\section{Hankel determinants and ratios of close partition numbers}

The hyperbolicity of a polynomial $P(X) = a_d X^d + a_{d-1}X^{d-1} + \ldots + a_0$ is equivalent to certain polynomial conditions in the coefficients $a_i$, which we now describe. If $\lambda_1, \ldots, \lambda_d$ are the roots of $P(X)$, let $S_k = \lambda_1^k + \ldots + \lambda_d^k$ denote the sum of $k$th powers of the roots. The $m \times m$ \textit{Hankel determinant associated to $P(X)$} is defined by
\begin{equation}\label{hmat}
\Delta_m (P(X)) := \left| \begin{matrix} S_0 & S_1 & \cdots &S_{m-1} \\ S_1 & S_2 & \cdots & S_{m} \\ \vdots & \vdots & & \vdots \\ S_{m-1} & S_m & \cdots & S_{2m-2} \end{matrix}\right| = \sum_{i_1 < \cdots < i_m} \prod_{a<b} (\lambda_{i_a} - \lambda_{i_b})^2.
\end{equation}
In addition, let
\[D_{d,m}(P(X)) = D_{d,m}(a_0, \ldots, a_{d}) := a_d^{2m-2} \cdot \Delta_m(P(X))\]
so that $D_{d,d}(a_0, \ldots, a_{d})$ is the discriminant of $P(X)$ and $D_{d,m}(a_0, \ldots, a_d)$ is a homogeneous polynomial of degree $2m-2$ in the coefficients $a_i$. A theorem of Hermite \cite{Her} says the hyperbolicity of $P(X)$ is equivalent to the condition $D_{d,m}(P(X)) \geq 0$ for all $m=2, \ldots, d$.

We will prove Theorems \ref{lowdegrees} and \ref{general} by showing that
\begin{equation*} %\label{D}
\mathcal{D}_{d,m}(n) := D_{d,m}\left(\frac{J^{d,n}_p(X)}{p(n)}\right) = D_{d,m}\left(1, {d \choose 1}\frac{p(n+1)}{p(n)}, {d \choose 2} \frac{p(n+2)}{p(n)}, \ldots, \frac{p(n+d)}{p(n)}\right) > 0
\end{equation*}
for each $m=2, \ldots, d$ and all $n$ greater than the claimed quantities. Note that $\mathcal{D}_{d,m}(n)$ approaches $0$ in the limit as $n \rightarrow \infty$, since $\lim_{n\rightarrow \infty} J_p^{d,n}(X)/p(n) = (X+1)^d$.  This fact is true because the partition ratios $\frac{p(n+j)}{p(n)} \rightarrow 1$ as $n \rightarrow \infty$ for any fixed $j$. A priori, this makes the sign of $\mathcal{D}_{d,m}(n)$ difficult to ascertain.

However, the results in \cite{GORZ} determine the rate at which $\mathcal{D}_{d,m}(n)$ approaches $0$ and the coefficient of the leading term.
More precisely, by the behavior of $\Delta_m$ under change of variable and \eqref{grozmain}, we know that
\[\lim_{n \rightarrow \infty} \frac{1}{\delta(n)^{m(m-1)}}\Delta_m\left(\frac{J_p^{d,n}(X)}{p(n)}\right) = \lim_{n\rightarrow \infty}\Delta_m\left(\frac{J_p^{d,n}(\delta(n) X - e^{-cw(n)/2})}{p(n)}\right) = \Delta_m(H_d(X)).\]
Equivalently in terms of $w=w(n)=1/\sqrt{c(n-1/24)}$ and $\mathcal{D}_{d,m}(n)$, we have
\begin{equation} \label{key}
\lim_{w\rightarrow 0} \frac{1}{w^{\frac{3}{2}m(m-1)}}\mathcal{D}_{d,m}(n) = \left(\frac{c}{\sqrt{2}}\right)^{m(m-1)} \Delta_m(H_d(X)).
\end{equation}
Because the Hermite polynomials have distinct, real roots, the term on the right is a positive constant.
Our strategy is to expand $\mathcal{D}_{d,m}(n)$ in powers of $w$ around zero, up to $w^{\frac{3}{2}m(m-1)}$. Because the above limit exists, we are guaranteed that all lower powers of $w$ cancel, and the coefficient of the $w^{\frac{3}{2}m(m-1)}$ term is the specified positive multiple of $\Delta_m(H_d(X))$. We then must find explicit bounds for the remaining terms that are tending to zero. 

To do this, we need to study ratios of close partition numbers.
 In terms of $w$, the Hardy-Ramanunjan asymptotic formula for the partition numbers \cite{HR} takes the form
 \[p(n) \sim F(w) := \frac{\pi^2}{6\sqrt{3}}(w^2 - w^3) e^{1/w}.\]
As observed in \cite{GORZ}, $w(n+j) =  \frac{w(n)}{\sqrt{1+cjw(n)^2}}$, so the function
\begin{equation} \label{R}
R(j,w) := \frac{F\left(\frac{w}{\sqrt{1+cjw^2}}\right)}{F(w)}=\frac{e^{\frac{cjw}{1+\sqrt{1+cjw^{2}}}}(\sqrt{1+cjw^2} -w)}{(1-w)(1+cjw^2)^{3/2}}
\end{equation}
closely approximates $p(n+j)/p(n)$. 

To bound the error of this approximation, we use Lehmer's error bound for Rademacher's convergent series for the partition function, in which $F(w)$ is the leading term. In what follows, $A_k(n)$ is a Kloosterman sum. The only property we need is $|A_1(n)|=|A_2(n)|=1$, so we do not define it here, instead referring the reader to \cite{Leh}.
\begin{Thm}[Lehmer]
Let $w=w(n) = 1/\sqrt{c(n-1/24)}$. For all $n \geq 1$, we have
\begin{equation} \label{rad}
p(n) = \frac{\pi^2}{6\sqrt{3}} w^2 \sum_{k=1}^N \frac{A_k(n)}{\sqrt{k}} \left((1-w)e^{1/kw} + (1+w)e^{-1/kw}\right)+B(n,N),
\end{equation}
where
\[|B(n,N)|<\frac{\pi^2N^{-2/3}}{\sqrt{3}}\left( N^3w^3\sinh\left(\frac{1}{Nw}\right) + \frac{1}{6} - N^2w^2\right) < 
\frac{\pi^2N^{-2/3}}{\sqrt{3}}\left( N^3w^3\frac{e^{1/Nw}}{2} + \frac{1}{6}\right)
.
\]
\end{Thm}
In order for us to state precisely how well $R(j,w)$ approximates $p(n+j)/p(n)$, let
\[L(w) := \frac{1+21w}{1-w} \cdot e^{-1/2w} + \frac{e^{-1/w}}{w^2-w^3}.\]

\begin{Lemma} \label{bound1}
For all $n \geq 1$, we have
\[\left| \frac{p(n+j)}{p(n)} - R(j,w)\right| \leq R(j,w) \frac{2L(w)}{1-L(w)} \sim 2 e^{-1/2w}.\]
\end{Lemma}
\begin{proof}
Let $E(w(n)) = p(n) - F(w(n))$. The function $F(w)$ appears in the $k=1$ term of \eqref{rad}. Gathering the rest of that term, the $k=2$ term, and the Lehmer's bound on $|B(n,2)|$ we find
\begin{align*}
|E(w)| &\leq \frac{\pi^2}{6\sqrt{3}}\left((w^2+w^3)e^{-1/2w} + (w^2 - w^3+12\cdot 2^{5/6} w^3)e^{1/2w} + 2^{-7/6}\right) \\
&\leq \frac{\pi^2}{6\sqrt{3}}\left((w^2+21w^3)e^{1/2w}+1\right),
\end{align*}
where in the last line we have used that $w \leq 1/\sqrt{c}$. Hence, $|E(w)/F(w)| \leq L(w)$. Noting that the function $L(w)$ is increasing in $w$ for $0 < w \leq 1/\sqrt{c}$, it follows that
\begin{align*}
\left | \frac{p(n+j)}{p(n)} - \frac{F(w(n+j))}{F(w(n))} \right |&=\frac{F(w(n+j))}{F(w(n))} \left| \frac{1 + \frac{E(w(n+j))}{F(w(n+j))}}{1 + \frac{E(w)}{F(w)}} - 1 \right| \\
&=R(j,w)\left|\frac{\frac{E(w(n+j))}{F(w(n+j))}-\frac{E(w(n))}{F(w(n))}}{1 + \frac{E(w)}{F(w)}} \right | \leq R(j,w) \frac{2L(w)}{1-L(w)}. \qedhere
\end{align*}
\end{proof}
To study the behavior $p(n+j)/p(n)$ for large $n$, we want to study $R(j,w)$ near $w=0$. To this end, let $A_{s}(j,w)$ be the degree $s-1$ Taylor polynomial of $R(j, w)$.
Applying Lemma \ref{bound1} and Taylor's theorem, we immediately obtain the following. 

\begin{Lemma} \label{bound}
Let $n \geq 1$ and suppose $w=1/\sqrt{c(n-1/24)} \in [0, \epsilon]$ for some $0 < \epsilon \leq 1/\sqrt{c}$. Then we have
\[\frac{p(n+j)}{p(n)} = A_{s}(j,w) + E_{s}(j,w) w^s,\]
where
\begin{equation}
|E_{s}(j,w)| \leq \frac{1}{s!} \cdot \sup_{x \in [0, \epsilon]} \left |R^{(s)}(j,x) \right| + \sup_{x \in [0,\epsilon]} \left|R(j,x) \frac{2L(x)}{x^s(1-L(x))} \right|.
\end{equation}
\end{Lemma}

\section{Proof of Theorems \ref{lowdegrees} and \ref{con}}

We now prove Theorem \ref{lowdegrees} by bounding the error terms that accumulate from approximating $p(n+j)/p(n)$ by the Taylor polynomials $A_s(j,w)$ in the polynomial expression for $\mathcal{D}_{d,m}(n)$. This allows us to reduce to checking finitely many cases.
\begin{proof}[Proof of Theorem \ref{lowdegrees}]
Using the Newton-Girard identities to write the power sums of the roots in terms of the elementary symmetric functions, one can generate symbolic expressions for the polynomials $D_{d,m}(a_0, \ldots, a_d)$ in terms of $a_0, \ldots, a_n$.
To obtain $\mathcal{D}_{d,m}(n)$, we substitute
\[{d \choose j}(A_{10}(j,w) + E_j w^{s})\]
in for $a_j$ in these polynomials, introducing $E_j$ as a variable. This gives rise to a polynomial expression in $w$ whose coefficients are polynomials in $E_j$. It turns out that all coefficients of $w^i$ for $i < k= \frac{3}{2}m(m-1)$ vanish in this expression. In addition, dividing through by $w^{k}$ gives rise to an expression of the form
\[\mathcal{D}_{d,m}(w) = c_0 + c_1 w + c_2(E_1,\ldots, E_d) w^2 + \ldots + c_{(2m-2)s-k}(E_1, \ldots, E_d)w^{(2m-2)s-k},\]
where $c_0$ and $c_1$ are positive constants.

We then use Mathematica to calculate the upper bound on $E_j=E_{10}(j,w)$ for $w \in [0, \epsilon]$ given in Lemma \ref{bound}, where we choose
\[\epsilon = 0.021,0.0163,0.0081 \qquad \text{ for $d=3,4,5$ respectively.}\] 
 From these, we can obtain a lower bound $-c_i' \leq c_i(E_1, \ldots, E_d)$ for each $i\geq 2$, giving rise to an expression of the form
\[\mathcal{D}_{d,m}(w) \geq c_0 + c_1 w - c_2'w^2 - \ldots - c_{(2m-2)s-k}' w^{(2m-2)s-k}.\]
Moreover, we can arrange for each of the $c_i'$ above to be nonnegative so that the function on the right crosses zero at most once in the interval $[0, \epsilon]$. For our chosen values of $\epsilon$, evaluating the right-hand side at $w=\epsilon$ is positive, so $\mathcal{D}_{d,m}(w) > 0$ for all $1 \leq m \leq d$ and $w \leq \epsilon$. Equivalently, $J_{p}^{d,n}(X)$ is hyperbolic for all $n \geq \frac{1}{c \epsilon^2}+\frac{1}{24}$. Using the values of $\epsilon$ listed above, this shows $J_p^{3,n}(X)$ is hyperbolic for all $n > 344$, $J_p^{4,n}(X)$ is hyperbolic for all $n > 572$ and $J_p^{5,n}$ is hyperbolic for all $n > 2316$. Checking the finite number of remaining possible counter examples directly now proves the theorem.
Annotated Sage and Mathematica code to implement the full procedure described above appears in the appendix.
\end{proof}

\begin{rmk}
With our chosen parameters, the total run time of this procedure is about $15$ minutes.
We note that by increasing the number of terms $s$ that we take in the Taylor expansion of $R(j, w)$, the number of cases one needs to check directly can be brought down. However, this increases total run time, as checking more particular cases directly is faster than carrying out the more complex symbolic manipulations. For example, when $d = 5$, by increasing $s$ to $16$, one may increase $\epsilon$ to $0.013$, corresponding to checking $n=899$ cases directly, but this has a total run time of about an hour.

\end{rmk}
\begin{rmk}
For $d \geq 6$ one would need to keep more than $s=10$ terms in order to see the cancellation of lower order terms in $w$ take place.  The main obstruction of applying this method in higher degrees is tracking the increasing number of error terms in the increasingly complex symbolic expressions for $\mathcal{D}_{d,m}(n)$. A code for $d=6$ with $s = 16$ did not finish within $36$ hours when run on a laptop.
\end{rmk}

Taylor expanding $R(j,w)$ and symbolically keeping track of errors can be used to prove inequalities about other polynomial equations involving ratios of close partition numbers.  We now prove Theorem \ref{con} using this idea.
\begin{proof}[Proof of Theorem \ref{con}]
Setting $a_i = p(n+i)/p(n)$ we can rewrite \eqref{chen} as
\[0 <  \left(1+\frac{\pi^4}{9}w^3\right)(a_1^2-a_{-1}a_{1}a_{2})^2 - 4a_1^2(1-a_{-1}a_{1})(a_1^2-a_2). \]
We follow the same procedure and notation as in the proof of Theorem \ref{lowdegrees}, taking $s = 6$ and $\epsilon = 0.013$.
Substituting $a_i = A_{6}(i,w) + E_i w^{6}$ into the right-hand side above gives rise to a polynomial expression in $w$ with coefficients that are polynomials in the $E_i$, where the first term is a positive constant times $w^{10}$. We then minimize all the coefficients as before, using the bounds on $|E_i|$ from Lemma \ref{bound}. This leaves us with an expression of the form 
\[w^{10}\left(\frac{25}{729} \pi^{12} - x(w)\right) \leq \left(1+\frac{\pi^4}{9}w^3\right)(a_1^2-a_{-1}a_{1}a_{2})^2 - 4a_1^2(1-a_{-1}a_{1})(a_1^2-a_2),\]
where $x(w)$ is a strictly increasing polynomial in $w$. Evaluating the left-hand side at $w=\epsilon$ yields a positive number, so the right-hand side is positive for all $w \in [0, \epsilon]$. Equivalently, the proposition holds for all $n > 900$. Checking all $n \leq 900$ directly completes the proof.
\end{proof}

\section{Bounds for general $d$}
The polynomial $\mathcal{D}_{d,m}(n)$ we wish to study is homogeneous of degree $2m-2$ in the coefficients of $J^{d,n}_p(X)/p(n)$ and homogeneous of degree $m(m-1)$ in its roots. That is, it has the form
\begin{equation}\label{mon}
\mathcal{D}_{d,m}(n) = \sum_{i_1 + \ldots + i_{2m-2}=m(m-1)} A_{i_1, \ldots, i_{2m-2}} \cdot \prod_{k=1}^{2m-2} {d \choose i_k} \frac{p(n+d-i_k)}{p(n)},
\end{equation}
where the $A_{i_1, \ldots, i_{2m-2}}$ are constants.
To bound errors when we expand in terms of $w$, we find bounds on the derivatives $R^{(s)}(j, w)$ for $w$ in the interval $[0, \epsilon]$, where $\epsilon:=(3d)^{-12d}(50d)^{-\frac{3}{2}d^2}$, corresponding to our eventual bound on $N(d)$. For convenience, let $t = t(j) :=cj$.

\begin{Lemma} \label{Deriv}
Assume that $w \in [0, \epsilon]$ with $\epsilon$ as above.  Then
\begin{equation}
\left | R^{(m)}(j,w) \right| \leq m! \binom{m+3}{3} e^{g(\epsilon)} (4 e^{2t\epsilon} t)^{m},
\end{equation}
where $g(\epsilon) = \frac{t\epsilon}{1 + \sqrt{1+t\epsilon^2}}$.
\end{Lemma}

\begin{proof}
The idea of the proof is to use the product rule to split up $R(j,w)$ into four more manageable parts and use Fa\`{a} di Bruno's formula for iterated applications of the chain rule to evaluate each part as needed.  This formula says that for differentiable functions $f(x)$ and $g(x)$, we have
\begin{equation} \label{Bruno}
\frac{d^{n}}{dx^{n}} f \left(g(x) \right) = \!\!\!\!\! \sum_{m_{1} + 2 \cdot m_{2} + \cdots + n \cdot m_{n} = n} \frac{n!}{m_{1}! \cdots m_{n}!} f^{(m_{1} + m_{2} + \cdots + m_{n})} \left(g(x) \right) \prod_{j=1}^{n} \left( \frac{g^{(j)}(x)}{j!} \right)^{m_{j}}.
\end{equation} 
Let
\begin{align*}
A &= A(t,w) := e^{\frac{tw}{1 + \sqrt{1+tw^2}}} & B &= B(t,w) := \sqrt{1+tw^2} - w, \\
C &= C(t,w) := \frac{1}{1-w} & D &= D(t,w) := \frac{1}{(1+tw^2)^{3/2}},
\end{align*}
so that
\begin{equation} \label{Rderiv}
R^{(m)}(j,w) = \!\!\!\! \sum_{m_{1} + m_{2} + m_{3} + m_{4} = m} \frac{m!}{m_{1}! \cdots m_{4}!} \left(\frac{d^{m_{1}}A}{dw^{m_{1}}} \right) \cdot \left(\frac{d^{m_{2}} B}{dw^{m_{2}}} \right) \cdot \left(\frac{d^{m_{3}} C}{dw^{m_{3}}} \right) \cdot \left(\frac{d^{m_{4}} D}{dw^{m_{4}}} \right).
\end{equation}
We will focus on $A$ first.  Let $f(w) = e^w$ and $g(w) = \frac{tw}{1 + \sqrt{1+tw^2}}$. By (\ref{Bruno}), we have
\begin{equation} \label{dA}
\frac{d^nA}{dw^n}=\frac{d^n}{dw^n} f(g(w)) = \sum_{m_{1} + 2\cdot m_{2} + \cdots + n \cdot m_{n} =n} \frac{n!}{m_{1}! \cdots m_{n}!} e^{g(w)} \prod_{i=1}^{n} \left( \frac{g^{(i)}(w)}{i!} \right)^{m_{i}}.
\end{equation}
By the product rule, it is easy to see that
\begin{equation} \label{g}
g^{(i)}(w) = tw \left(\frac{d}{dw} \right)^{i} \frac{1}{1+\sqrt{1+tw^2}} + it \left(\frac{d}{dw} \right)^{i-1} \frac{1}{1 + \sqrt{1+tw^2}}.
\end{equation}
Next, let $g_{*}(w) := \frac{1}{1+\sqrt{1+tw^2}}$ and let $\alpha(k) := \left(\frac{d}{dw} \right)^{k} \sqrt{1+tw^2}$. We use (\ref{Bruno}) again to show
\begin{equation} \label{g*}
g_{*}^{(i)}(w) = \sum_{r_{1} + \cdots + i \cdot r_{i} =i} \frac{i!}{r_{1}!\cdots r_{i}!} \frac{(-1)^{r_{1} + \cdots + r_{i}} (r_{1} + \cdots + r_{i})!}{(1+ \sqrt{1+tw^2})^{r_{1} + \cdots + r_{i} +1}} \prod_{k=1}^{i} \left( \frac{\alpha(k)}{k!} \right)^{r_{k}}.
\end{equation}
Using (\ref{Bruno}) once more we have 
\begin{align*}
\alpha(k) 
%&= \sum_{s_{1} + \cdots + k \cdot s_{k}=k} \binom{k}{s_{1}, \dots, s_{k}} \binom{1/2}{s_{1} + \cdots + s_{k}} \frac{1}{(1+tw^2)^{s_{1} + \cdots + s_{k} - \frac{1}{2}}} \prod_{v=1}^{k} \left( \frac{(1+tw^2)^{(v)}}{v!} \right)^{s_{v}} \\
&= \sum_{s_{1} + 2 s_{2} = k} \frac{k!}{s_{1}! s_{2}!} \binom{\frac{1}{2}}{s_{1} + s_{2}} \frac{(2tw)^{s_{1}} t^{s_{2}}}{(1+tw^2)^{s_{1} + s_{2} - \frac{1}{2}}} \leq k! e^{2tw} t^{k}.
\end{align*}
%From this we can see that $| \alpha(k) | \leq k! (\lambda t)^{k}$ where $\lambda^{k}$ is an upper bound on $\sum_{s_{1} + 2 \cdot s_{2} = k} \frac{1}{s_{1}! \cdot s_{2}!} \binom{1/2}{s_{1} + s_{2}} (2tw)^{s_{1}}$.  Note that $\lambda$ does not need to depend on $k$ because  we can easily bound the sum by $e^{2tw}$.  
We can plug this back into (\ref{g*}) to find that
\[g_*^{(i)}(w) \leq i! (e^{2tw} t)^i \sum_{r_{1} + \cdots + i \cdot r_{i} =i}\frac{(r_{1} + \cdots + r_{i})!}{r_{1}! \cdots r_{i}!} \leq i!(2e^{2tw}t)^{i},\]
where we used the fact that the sum is counting the number of ordered partitions of $i$. Next, we plug this into (\ref{g}) and use the fact that $tw \leq 1$ to find $\left| g^{(i)}(w) \right| \leq i! \cdot 2(2 \lambda t)^{i}$.
Finally, we are able to plug this into (\ref{dA}) to find that
\begin{equation} \label{A}
\left | \frac{d^nA}{dw^n} \right | \leq n! e^{g(w)} (2 e^{2tw} t)^{n} \cdot \!\!\!\!\sum_{m_{1} + \cdots + n \cdot m_{n} =n} \frac{2^{m_{1} + \cdots + m_{n}}}{m_{1}! \cdots m_{n}!} \leq n! e^{g(w)} (4 e^{2tw} t)^{n}.
\end{equation}
Next, it is easy to show that 
\begin{equation} \label{B}
\left | \frac{d^nB}{dw^n} \right| \leq \left| \alpha(n) \right| \leq n! (e^{2tw} t)^{n},
\end{equation}
and
\begin{equation} \label{C}
\left|\frac{d^nC}{dw^n}\right| = \frac{n!}{(1-w)^{n+1}} \leq n! (e^{2tw} t)^{n}.
\end{equation}
Lastly, we have 
\begin{equation} \label{D}
 \left|\frac{d^nD}{dw^n}\right| \leq \sum_{r_{1} + \cdots + n \cdot r_{n} =n} \frac{n!}{r_{1}! \cdots  r_{n}!} \frac{ (\frac{3}{2})_{r_{1} + \cdots + r_{n}}}{(1+tw^2)^{\frac{3}{2}+ r_{1} + \cdots + r_{n}}} \prod_{k=1}^{n} \left( \frac{|\alpha(k)|}{k!} \right)^{r_{k}} \leq n! (2 e^{2tw} t)^{n}.
 \end{equation}
where $(x)_{n} := x(x+1) \cdots (x+n-1)$ is the rising factorial.  
Finally, we substitute the bounds in equations \eqref{A}, \eqref{B}, \eqref{C}, and \eqref{D} back into \eqref{Rderiv} and use the fact that the sum over $m_1 + \ldots + m_4 = m$ contains $\binom{m+3}{3}$ terms.
\end{proof}

Given some $\underline{i} = (i_1, \ldots, i_{2m-2})$ with $i_1+\ldots + i_{2m-2} = m(m-1)$, let $T_{d,m}(\underline{i}; w)$ be the degree $\frac{3}{2}m(m-1)$ Taylor polynomial of $\prod_{k=1}^{2m-2} R(d-i_k, w)$.
\begin{Lemma} \label{TE}
Suppose $w \in [0,\epsilon]$. Then
\begin{equation}\label{tay}
\prod_{k=1}^{2m-2} \frac{p(n+d-i_k)}{p(n)} = T_{d,m}(\underline{i};w) + E_{d,m}(\underline{i};w)w^{\frac{3}{2}m(m-1)+1}
\end{equation}
where
\[|E_{d,m}(\underline{i};w)| \leq e^2(3d)^{10d-10}(4cd)^{\frac{3}{2}d^2} + 8m \cdot 6^{2m} \leq 2e^2(3d)^{10d-10}(4cd)^{\frac{3}{2}d^2}. \]
\end{Lemma}

\begin{proof}
By Lemma \ref{bound1}, we can write
\[\prod_{k=1}^{2m-2} \frac{p(n+d-i_k)}{p(n)} = \prod_{k=1}^{2m-2} R(d-i_k,w)(1 + U_k(w)) = \prod_{k=1}^{2m-2}R(d-i_k,w) + U(w),\]
where
\begin{align*}
|U(w)| &\leq \prod_{k=1}^{2m-2}R(d-i_k,w) \left(\left(1 + \frac{2L(w)}{1-L(w)}\right)^{2m-2}  - 1\right) \\
&\leq 2^{2m-2} \cdot (2m-2)\cdot 3^{2m-2} \cdot \frac{2L(w)}{1-L(w)} \leq 8m \cdot 6^{2m} \cdot e^{-1/2w}.
\end{align*}
Let $s = \frac{3}{2}m(m-1)+1$. Note also that we can easily bound
\[\frac{e^{-1/2w}}{w^s} \leq \frac{e^{-1/2\epsilon}}{\epsilon^s} \leq \exp\left(\frac{3}{2}d^2 \left(2d \log(3d) + \frac{3}{2}d^2\log(50d)\right) - \frac{1}{2}(3d)^{12d}(50d)^{\frac{3}{2}d^2}\right) < 1.\]
Meanwhile, from Lemma \ref{Deriv} and the product rule, we know that 
\begin{align*}
&\frac{1}{s!} \left | \frac{d^s}{dw^s} \prod_{k=1}^{2m-2}R(d-i_{k},w) \right | \leq e^{(2m-2)g(\epsilon)}(4 e^{2cd\epsilon} cd)^{s} \\
 &\qquad\qquad \qquad\qquad\qquad\qquad\qquad \times \!\!\!\!\!\sum_{n_{1} + \cdots + n_{2m-2} = \frac{3}{2}m(m-1) +1} \binom{n_{1} +3}{3} \cdots \binom{n_{2m-2}+3}{3}.\end{align*}
The largest term in the sum on the right hand side occurs if each $n_{i}$ is equal, which is in turn bounded by replacing each $n_i$ with  $m \geq \frac{\frac{3}{2}m(m-1) +1}{2m-2}$. Counting the number of terms, we see that the sum is bounded above by
\begin{align*}
 \binom{\frac{3}{2}m(m-1)+ 2m-2}{2m-3} \cdot \binom{m+3}{3}^{2m-2} \leq (2m^2)^{2m-2} \cdot \left(\frac{3}{2} m^3\right)^{2m-2} = (3m)^{10m-10}. 
\end{align*}
This shows that
\begin{align*}
\left|\prod_{k=1}^{2m-2}R(d-i_k,w) - T_{d,m}(\underline{i};w) \right| &\leq e^{(2m-2)g(\epsilon)}(4e^{2d\epsilon}cd)^s(3m)^{10m-10} \cdot w^s \\
&\leq e^2(3d)^{10d-10}(4cd)^{\frac{3}{2}d^2} \cdot w^s. \qedhere
\end{align*}

In order to finish bounding the monomials in equation \eqref{mon} we need the following result. We include the extra factor out front because of how it enters in equation \eqref{key}.
\begin{Lemma} \label{binomial}
Suppose $0 \leq m \leq d$ and $i_1 + \ldots + i_{2m-2}= m(m-1)$ for positive integers $i_k$. Then we have
 \begin{equation}
 \left | \left( \frac{\sqrt{2}}{c} \right)^{m(m-1)} \prod_{k=1}^{2m-2} \binom{d}{i_{k}} \right | \leq \left( e^{\frac{4e}{c^{2}}} \right)^{d^{2}}.
 \end{equation}
 \end{Lemma}
 \begin{proof}
The product $\prod_{k=1}^{2m-2} \binom{d}{i_{k}}$ is maximized when all $i_{k}$ are equal (i.e. $i_{k} = \frac{m}{2}$).  Using standard bounds on binomial coefficients, we therefore have $\prod_{k=1}^{2m-2} \binom{d}{i_{k}} \leq \left( \frac{2ed}{m} \right)^{m(m-1)}$.  For $0 \leq m \leq d$, the function $\left( \frac{2 \sqrt{2} ed}{cm} \right)^{m^2}$ achieves its maximum at $m=\frac{2 \sqrt{2e} d}{c}$.  Thus 
 \[ \left | \left( \frac{\sqrt{2}}{c} \right)^{m(m-1)} \prod_{k=1}^{2m-2} \binom{d}{i_{k}} \right | \leq \left | \left( \frac{2 \sqrt{2}ed}{cm} \right)^{m^2} \right | \leq \left( e^{\frac{4e}{c^{2}}} \right)^{d^2}. \qedhere\]
 \end{proof}

We now have bounds on the errors of our approximations of each monomial in \eqref{mon}. We also must bound the number of such terms that appear in this equation for $\mathcal{D}_{d,m}(n)$.
  \begin{Lemma} \label{monomials}
Suppose $n > (3d)^{24d}(50d)^{3d^2}$ and let $A_{i_1, \ldots, i_{2m-2}}$ be as in \eqref{mon}. Then
 \begin{equation}
\sum_{i_1, \ldots, i_{2m-2}} |A_{i_1, \ldots,i_{2m-2}}| \leq m! (m-1)^{m} 2^{m^2-2} \leq d^{2d} \cdot 2^{d^{2}}.
 \end{equation}
 \end{Lemma}
 
 \begin{proof}
 By the Newton-Girard identities, the power sums $S_{k}$ in the matrix in \eqref{hmat} can be written as a sum of at most
 \[k \sum_{r_{1} + \cdots + k \cdot r_{k} = k} \frac{(r_{1} + \cdots + r_{k} -1)!}{r_{1}!  \cdots r_{k}!} \leq k 2^{k-1}\]
monomials in the coefficients of our polynomial.
 The determinant of the matrix in \eqref{hmat} is made up of a sum of at most $m!$ monomials of the form 
 \[\prod_{\ell=1}^{m} S_{i_{\ell}} \qquad \text{where $i_{1} + \cdots + i_{m} = m(m-1)$.}\]
 Plugging in the elementary symmetric functions for each $S_{i_\ell}$ in this product and expanding will express each of these ``$S$-monomials" as a sum of at most
 \[\prod_{\ell=1}^{m} i_{l} 2^{i_{\ell} -1} \leq (m-1)^{m} 2^{m(m-2)}\]
monomials in the coefficients. To obtain $\mathcal{D}_{d,m}(n)$ from this, we must multiply by $(\frac{p(n+d)}{p(n)})^{2m-2}$. Since $n$ is so large, we easily have $p(n+d)/p(n) \leq 2$, for example by using Lemma \ref{bound1} with $s=1$. Multiplying together the factors discussed above gives the result.
 \end{proof}
 
The last ingredient we need to prove Theorem \ref{general} is a lower bound on the Hankel determinants of Hermite polynomials.
 \begin{Lemma} \label{discriminant}
For each $m \leq d$, we have
$ \Delta_{m}(H_{d}(X)) \geq 1$.
 \end{Lemma}
 
 \begin{proof}
 We know $\Delta_{m}(H_{d}(X)) = \sum_{i_{1}< \cdots < i_{m}} \prod_{a<b} ( \lambda_{i_{a}} - \lambda_{i_{b}})^{2}$ so by the inequality of the arithmetic and geometric mean 
 \begin{align*}
 \Delta_{m}(H_{d}(X)) &\geq \binom{d}{m} \prod_{i_{1}< \cdots < i_{m}}  \left( \prod_{a<b} ( \lambda_{i_{a}} - \lambda_{i_{b}})^{2} \right)^{\frac{1}{\binom{d}{m}}} = \binom{d}{m} \left( \prod_{j<k} ( \lambda_{j} - \lambda_{k})^{2 \binom{d-2}{m-2}} \right)^{\frac{1}{ \binom{d}{m}}} \\
 &= \binom{d}{m} \Delta_{d}(H_{d}(X))^{\frac{m(m-1)}{d(d-1)}}.
 \end{align*}
 By Theorem 6.71 of \cite{Szego}, and the fact that $a_{d}(H_{d}(X)) = 2^{d}$, we have
 \[\Delta_{d}(H_{d}(X)) = \frac{\mathrm{Disc}(H_d(X))}{2^{2d(d-1)}} = 2^{-\frac{d(d-1)}{2}} \prod_{\nu=1}^{d} \nu^{\nu} \geq 1,\]
 so the result follows.
 \end{proof}

Proving Theorem \ref{general} is now just a matter of collecting and bounding all of the higher order terms from expanding $\mathcal{D}_{d,m}(n)$ in terms of $w$.
\begin{proof}[Proof of Theorem \ref{general}]
Suppose $n > (3d)^{24d}(50d)^{3d^2}$ so that $w(n) \in [0, \epsilon]$.
By \eqref{key}, we have
\begin{align*}
\frac{\mathcal{D}_{d,m}(n)}{w^{\frac{3}{2}m(m-1)}} &=\sum_{i_1, \ldots, i_{2m-2}=m(m-1)} \frac{A_{i_1, \ldots, i_{2m-2}}}{w^{\frac{3}{2}m(m-1)}} \cdot \prod_{k=1}^{2m-2}{d \choose i_k}\left(T_{d,m}(\underline{i};w)+E_{d,m}(w)w^{\frac{3}{2}m(m-1)+1}\right) \\
&= \left(\frac{c}{\sqrt{2}}\right)^{m(m-1)}\Delta_m(H_d(X)) + w \cdot \mathcal{E}_{d,m}(w),
\end{align*}
where by Lemmas \ref{monomials}, \ref{binomial}, and \ref{TE},
\begin{align*}
\left(\frac{\sqrt{2}}{c}\right)^{m(m-1)} \cdot |\mathcal{E}_{d,m}(w)| \cdot w &\leq d^{2d} \cdot 2^{d^2} \cdot \left(e^{\frac{4e}{c^2}}\right)^{d^2} \cdot 2e^2 (3d)^{10d-10}(4cd)^{\frac{3}{2}d^2} \cdot w \\
&< (3d)^{12d}(50d)^{\frac{3}{2}d^2} \cdot w \leq 1.
\end{align*}
Since $\Delta_m(H_d(X)) \geq 1$, it follows that $\mathcal{D}_{d,m}(n) > 0$ and therefore, $J_{p}^{d,n}(X)$ is hyperbolic.
 \end{proof}

\end{proof}

\section*{Appendix}
The Sage and Mathematica code below implements the procedure described in the proof of Theorem \ref{lowdegrees}.

\vspace{.1in}
Sage code
{\tiny \begin{lstlisting}[language=Python]
epsilon_list=[0,0,0.0295,0.021,0.0163,0.0081,0.001] #list of our epsilon choices
error_list=[0,0,[0,12719.9+1.59552*10^8,328255+1.7476*10^8],[0,10559.2+4.30607*10^6,328255+4.
60022*10^6,3.77919*10^6+4.91402*10^6],[0,9026.37+51727.4,328255+54478.9,3.77919*10^6+57374.2,
1.75707*10^7+60420.8],[0,5893.44+1.54878*10^-6,328255+1.58991*10^-6,3.77919*10^6+1.63212*10^-
6,1.75708*10^7+1.67544*10^-6,5.37043*10^7+1.71991*10^-6]] #from Mathematica and Lemma 2.3

#build symbolic expressions for Hankel determinants in terms of power sums s_i
S.<s0,s1,s2,s3,s4,s5,s6,s7,s8>=PolynomialRing(QQ)
ss=[s0,s1,s2,s3,s4,s5,s6,s7,s8]
Matrices=[matrix([ [ss[k] for k in [j..j+i-1]] for j in [0..i-1] ]) for i in [0..5] ]
MM=[M.determinant() for M in Matrices] #Hankel determinant in terms of S_i
AA.<a0,a1,a2,a3,a4,a5>=PolynomialRing(QQ) #the coefficients aj of a polynomial
aa=[a0,a1,a2,a3,a4,a5]

var('w,p,j') #p=pi
c=2*p^2/3
s=10 #point of bounding errors -- see Remark following the proof of Theorem 1.1
#define the function R(j,w) that approximates p(n+j)/p(n)
R=-exp(j*c*w/(1 + sqrt(1 + j*c*w^2)))*(sqrt(1 + j*c*w^2) - w)/((w - 1)*(1 + j*c*w^2)^(3/2))
A=R.series(w,s).truncate() #degree s-1 Taylor polynomial
T.<E1,E2,E3,E4,E5,w,p,j>=PolynomialRing(QQ)
EE=[0,E1,E2,E3,E4,E5]
A=T(A) #put A in the polynomial ring

def collect_errors(c,err): #minimizes c, given list of bounds on |E_i|
	M=c.monomials()
	C=[a.n() for a in c.coefficients()]
	l=len(C)
	to_sub= dict((EE[i],err[i]) for i in [1..len(err)-1])
	new=[]
	for i in [0..l-1]:
		if M[i].degree(E1)==M[i].degree(E2)==M[i].degree(E3)==M[i].degree(E4)==M[i].
		degree(E5)==0: #a monomial with no error terms will stay the same
			new.append(C[i]*M[i].subs(p=RR(pi)).n())
		else:
			new.append(-abs(C[i]*M[i].subs(to_sub).subs(p=RR(pi)).n()))
	return min(0,sum(new))

for d in [2,3,4,5]:
	epsilon=epsilon_list[d]
	elem=[(-1)^i*aa[d-i]/aa[d] for i in [0..d]] #elem sym functs in roots of sum(a_iX^i)
	for i in [d+1..2*d-2]:
		elem.append(0)
	power_sums=[d] #list of power sums
	for k in [1..2*d-2]: #builds power sums recursively using Newton-Girard formulae
		power_sums.append((-1)^(k-1)*k*elem[k]+sum([(-1)^(k-1+i)*elem[k-i]*power_sums
		[i] for i in [1..k-1]]))
	hankel_list=[0,0] #polynomial expression for Hankel det in terms of coefficients aj
	for m in [2..d]:
		to_sub = dict( (ss[i],power_sums[i]) for i in [0..2*m-2] )
		D=MM[m].subs(to_sub)*aa[d]^(2*m-2)
		D=AA(D) #put D back in polynomial ring
		hankel_list.append(D)
	err=error_list[d]
	to_sub = dict((aa[i],binomial(d,i)*(A.subs(j=i)+EE[i]*w^s)) for i in [0..d])
	Delta_is_positive=[]
	for m in [2..d]:
		D=hankel_list[m] #D is D_{d,m}
		Delta=T(D.subs(to_sub)) #with A_s and symbolic errors plugged in
		k=3*m*(m-1)/2
		w=T(w)
		minimized_Delta = sum([Delta.coefficient({w:i}).subs(p=RR(pi)).n()*w^(i-k) fo
		r i in [0..k+1]]) + sum([collect_errors(Delta.coefficient({w:i}),err).n()*w^(
		i-k) for i in [k+2..(2*m-2)*s] ])
		if minimized_Delta.subs(w=epsilon).n() > 0:
			Delta_is_positive.append(m)
		else:
			print d,m,'choose smaller epsilon'
	if len(Delta_is_positive)==d-1:
		print 'For d =', d, 'J^{n,d} is hyperbolic for all n > ', floor(1/(c.subs(p=R
		R(pi))*epsilon^2)+1/24)
	else:
		print 'choose smaller epsilon'
\end{lstlisting}}

Mathematica code
{\tiny 
\begin{lstlisting}
c = 2/3*Pi^2;
R[j_, w_]:=-Exp[c*j*w/(1+Sqrt[1+c*j*w^2])](Sqrt[1+c*j* w^2]-w)/((w-1)(1+c*j* w^2)^(3/2))
L[w_]:=(1+21*w)/(1-w)*Exp[-1/(2*w)]+Exp[-1/w]/(w^2-w^3)
Do[Print[N[Maximize[{R[i, w]*L[w]/(w^10*(1 - L[w])),0<=w<=0.0295},w],30]],{i,1,2}]
Do[Print[N[Maximize[{R[i, w]*L[w]/(w^10*(1 - L[w])),0<=w<=0.021},w],30]],{i,1,3}]
Do[Print[N[Maximize[{R[i, w]*L[w]/(w^10*(1 - L[w])),0<=w<=0.0163},w],30]],{i,1,4}]
Do[Print[N[Maximize[{R[i, w]*L[w]/(w^10*(1 - L[w])),0<=w<=0.0081},w],30]],{i,1,5}]
Do[Print[N[Maximize[{Abs[D[R[i,w],{w,10}]]/Factorial[10],0<=w<=0.0295},w],30]],{i,1,2}]
Do[Print[N[Maximize[{Abs[D[R[i,w],{w,10}]]/Factorial[10],0<=w<=0.021},w],30]],{i,1,3}]
Do[Print[N[Maximize[{Abs[D[R[i,w],{w,10}]]/Factorial[10],0<=w<=0.0163},w],30]],{i,1,4}]
Do[Print[N[Maximize[{Abs[D[R[i,w],{w,10}]]/Factorial[10],0<=w<=0.0081},w],30]],{i,1,5}]
Do[If[CountRoots[PartitionsP[i+3]*x^3+3*PartitionsP[i+2]*x^2+3*PartitionsP[i+1]*x+Partitions
 P[i],x]<3,Print[i]],{i,94,344}]
Do[If[CountRoots[PartitionsP[i+4]*x^4+4*PartitionsP[i+3]*x^3+6*PartitionsP[i+2]*x^2+4*Partit
 ionsP[i+1]*x+PartitionsP[i],x]<4,Print[i]],{i,206,572}]
Do[If[CountRoots[PartitionsP[i+5]*x^5+5*PartitionsP[i+4]*x^4+10*PartitionsP[i+3]*x^3+10*Part
 itionsP[i+2]*x^2+5*PartitionsP[i+1]*x+PartitionsP[i],x]<5,Print[i]],{i,381,2105}]
\end{lstlisting}
}

\end{document}